\documentclass{amsart}
\usepackage{amssymb,amsmath, latexsym, mathabx, bm, dsfont,physics}
\usepackage{xcolor}
\usepackage{stmaryrd}
\usepackage{makecell,multirow,longtable}

\renewcommand{\baselinestretch}{\baselinestretch}
\renewcommand{\baselinestretch}{1.1}
\numberwithin{equation}{section}

\newtheorem{thm}{Theorem}[section]
\newtheorem{lem}[thm]{Lemma}
\newtheorem{cor}[thm]{Corollary}

\theoremstyle{definition}

\theoremstyle{remark}

\newcommand{\gen}{\text{gen}}

\newcommand{\n}{{\mathbb N}}
\newcommand{\z}{{\mathbb Z}}

\newcommand{\Mod}[1]{\ (\mathrm{mod}\ #1)}

\newcommand{\bx}{\bm x}
\newcommand{\by}{\bm y}

\begin{document}

\title{primitively universal quaternary quadratic forms}

\author[Jangwon Ju and et al] {Jangwon Ju, Daejun Kim, Kyoungmin Kim, Mingyu Kim, and Byeong-Kweon Oh}

\address{Department of Mathematics, University of Ulsan, Ulsan, 44610, Korea}
\email{jangwonju@ulsan.ac.kr}
\thanks{This work of the first author was supported by the National Research Foundation of Korea(NRF) grant funded by the Korea government(MSIT) (NRF-2019R1F1A1064037).}

\address{School of Mathematics, Korea Institute for Advanced Study, Seoul 02455, Korea}
\email{dkim01@kias.re.kr}
\thanks{This work of the second author was supported by a KIAS Indivisual Grant (MG085501) at Korea Institute for Advanced Study.}

\address{Department of Mathematics, Hannam University, Daejeon 34430, Korea}
\email{kiny30@hnu.kr}
\thanks{This work of the third author was supported by the National Research Foundation of Korea(NRF) grant funded by the Korea government(MSIT) (NRF-(2020R1I1A1A01055225)	}

\address{Department of Mathematics, Sungkyunkwan University, Suwon 16419, korea}
\email{kmg2562@skku.edu}
\thanks{This work of the fourth author was supported by the National Research Foundation of Korea(NRF) grant funded by the Korea government(MSIT) (NRF-2021R1C1C2010133).}

\address{Department of Mathematical Sciences and Research Institute of Mathematics, Seoul National University, Seoul 08826, Korea}
\email{bkoh@snu.ac.kr}
\thanks{This work of the fifth author was supported by the National Research Foundation of Korea(NRF) grant funded by the Korea government(MSIT) (NRF-2019R1A2C1086347 and NRF-2020R1A5A1016126).}



\begin{abstract}  A (positive definite and integral) quadratic form $f$ is said to be {\it universal} if it represents all positive integers, and is said to be   
{\it primitively universal} if it represents all positive integers primitively. 
We also say $f$ is {\it primitively almost universal} if it represents almost all positive integers primitively. Conway and Schneeberger proved (see \cite{b}) that there are exactly $204$ equivalence classes of universal quaternary quadratic forms. Recently, Earnest and Gunawardana proved in \cite{eg} that among $204$ equivalence classes of universal quaternary quadratic forms, there are exactly $152$ equivalence classes of primitively almost universal quaternary quadratic forms. 
In this article, we prove that there are exactly $107$ equivalence classes of primitively universal quaternary quadratic forms. 
We also determine the set of all positive integers that are not primitively represented by each of the remaining $152-107=45$ equivalence classes of primitively almost universal quaternary quadratic forms. 
 \end{abstract}

\maketitle

\section{Introduction}

A positive definite integral quadratic form of rank $k$
$$
f(x_1,x_2,\dots,x_k)=\sum_{i,j=1}^k f_{ij}x_ix_j \quad (f_{ij}=f_{ji} \in \z)
$$
is said to be {\it universal} if for any positive integer $n$, the diophantine equation 
$$
f(x_1,x_2,\dots, x_k)=n
$$ 
has an integer solution $(x_1,x_2,\dots,x_k) \in \z^k$.  After Lagrange's celebrated four square theorem, which implies that the quadratic form $x^2+y^2+z^2+t^2$ of rank $4$ is universal, a number of  quadratic forms of rank $4$ are proved to be universal(see, for example, \cite{ra} and \cite{w}). Note that there does not exist a positive definite integral universal quadratic form of rank $3$. In 2002,  Conway and Schneeberger proved that there are exactly $204$ equivalence classes of  positive definite integral universal quadratic forms of rank $4$. Furthermore, they  proved the so-called ``15-Theorem," which states that every positive definite integral quadratic form that represents 
$$
1,2,3,5,6,7,10,14, \  \text{and} \ 15
$$ 
is, in fact, universal, irrespective of its rank (see also \cite{b}).

Let $f(x_1,x_2,\dots,x_k)$ be a (positive definite and integral) quadratic form of rank $k$. We say $f$ represents an integer $n$ primitively if the diophantine equation
$$
f(x_1,x_2,\dots, x_k)=n
$$
has an integer solution $(x_1,x_2,\dots,x_k) \in \z^k$ such that $(x_1,x_2,\dots,x_k)=1$. We say $f$ represents $n$ primitively over $\z_p$ for some prime $p$ if the above diophantine equation has a solution  $(x_1,x_2,\dots,x_k) \in \z_p^k-(p\z_p)^k$. 
  If $f$ represents all positive integers primitively, then we say $f$ is {\it primitively universal}.  Note that any primitively universal quadratic form is, in fact, universal. Furthermore, any primitively universal quadratic form represents all integers primitively over $\z_p$ for any prime $p$.  Conversely, if a quadratic form with rank greater than or equal to $4$ represents all positive integers primitively over $\z_p$ for any prime $p$, then it is primitively almost universal, that is, it represents almost all positive integers primitively (for this, see \cite{ca}).  
 
 The first systematic study on primitively universal quadratic forms was initiated by Budarina in \cite{bu}.  In that article, she proved that among universal quaternary quadratic forms with odd discriminant, $39$ are primitively almost universal and $23$ are not primitively almost universal.  Note that if a quadratic form $f$ with class number one represents an integer $a$ primitively over $\z_p$ for any prime $p$, then it represents $a$ primitively (for this, see 102.5 of \cite{om}). Using this result, she could prove the primitively universalities of $3$  universal  quaternary diagonal quadratic forms including $x_1^2+x_2^2+x_3^2+2x_4^2$, and $7$ universal quaternary non-diagonal ones.  In 2015, Earnest and his collaborators proved in \cite{ekm} that there are exactly $27$ equivalence classes of primitively universal quaternary diagonal quadratic forms, including $3$ classes whose primitively universalities were proved by Budarina in \cite{bu}. The classification of all primitively universal quadratic forms over $\z_p$ for some prime $p$  including the case when $p=2$ was completed by Earnest and Gunawardana in \cite{eg}. They also proved that among $204$ equivalence classes of universal quaternary quadratic forms, there are exactly $152$ equivalence classes of primitively almost universal quaternary quadratic forms. Furthermore, they found $3$ equivalence classes of primitively universal quaternary quadratic forms including $x_1^2+2x_2^2+3x_2^2+5x_4^2+2x_2x_4$.  
 
In this article, we prove that there are exactly $107$ equivalence classes of primitively universal quaternary quadratic forms. Since the primitively universalities of $(3+7)+(27-3)+3=37$ quadratic forms have already been proved, it suffices to show the primitively universalities of the remaining $70$ candidates.  We also determine the set of positive integers that are not primitively represented by each of the remaining $152-107=45$ equivalence classes of primitively almost universal quaternary quadratic forms that are universal. Note that there are infinitely many equivalence classes of primitively almost universal quaternary quadratic forms.

Throughout this paper, the geometric language  of quadratic spaces and lattices will be adopted following \cite{om}.   Let $R$ be the ring of rational integers $\z$ or the ring of $p$-adic integers $\z_p$ for a prime $p$.
An $R$-lattice $L=R\bx_1+R\bx_2+\cdots+R\bx_k$ is a free $R$-module equipped with a non-degenerate symmetric bilinear map $B_L : L\times L \to R$.  The corresponding Gram matrix $\mathcal M_L$ with respect to the basis $\{\bx_i\}_{i=1}^k$ is defined by $(B_L(\bx_i,\bx_j))_{1\le i,j\le k}$. 
The corresponding quadratic form $f_L$ is defined by 
$$
f_L(x_1,x_2,\dots,x_k)=\sum_{i,j=1}^k B_L(\bx_i,\bx_j)x_ix_j.
$$
If $\mathcal M_L$ is diagonal, then we simply write 
$$
L \simeq \langle Q_L(\bx_1),Q_L(\bx_2),\dots,Q_L(\bx_k)\rangle,
$$
where $Q_L(\bx)=B_L(\bx,\bx)$ is the corresponding quadratic map. We simly write
$$
B(\bx,\by)=B_L(\bx,\by) \quad \text{and} \quad Q(\bx)=Q_L(\bx),
$$
if no confusion arises. 

Let $L$ be a $\z$-lattice of rank $k$. We say $L$ is primitive if $\mathfrak s(L)=\z$. Here the scale $\mathfrak s(L)$ of $L$ is defined by the ideal generated by $\{B(\bx,\by): \bx,\by \in L\}$ in $\z$.  The norm $\mathfrak n(L)$ of $L$  is defined by the ideal generated by $\{Q(\bx) : \bx \in L\}$ in $\z$. 
The discriminant $dL$ of $L$ is defined by the determinant of the Gram matrix of $L$.
For any prime $p$, we define $L_p=L\otimes \z_p$, which is considered as a $\z_p$-lattice.  We always assume that every $R$-lattice $L$ is {\it integral} in the sense that the scale $\mathfrak{s}L$ of $L$ is contained in  $R$ and is positive definite when $R=\z$, that is, the corresponding Gram matrix  $\mathcal M_L$ is positive definite.

For the $\z$-lattice $L$, we say a vector $\bx \in L$ is primitive if there is a basis $\mathfrak B$ of $L$ that contains $\bx$. The set of primitive vectors of $L$ is denoted by $L^*$. 
For a positive integer $a$,  we say $a$ is represented by $L$ if there is a vector $\bx \in L$ such that $Q(\bx)=a$.  We say $a$ is primitively represented by $L$ if there is a vector $\bx \in L^*$ such that $Q(\bx)=a$. The set of all positive integers that are (primitively) represented by $L$ is denoted by $Q(L)$ ($Q(L^*)$, respectively).  We also define $E(L)=\mathbb N-Q(L)$  and $E(L^*)=\mathbb N-Q(L^*)$, where $\mathbb N$ is the set of positive integers. 
We denote by  $Q(\gen(L))$ the set all positive integers that are represented by $L_p$ over $\z_p$ for any prime $p$. It is well known that $Q(\gen(L))$ is exactly equal to the set of all positive integers that are represented by a $\z$-lattice in the genus of $L$.    We say $L$ is (primitively) universal if $Q(L)=\mathbb N$ ($Q(L^*)=\mathbb N$, respectively).

For a positive integer $u$ and a nonnegative integer $r$, we define
$$
A_{u,r}=\{ uk+r : k \in \n \cup \{0\}\}.
$$

Throughout this article, $(a,b,c,d,e,f)$ indicates the quaternary quadratic form $x^2+ay^2+bz^2+cw^2+dzw+eyw+fyz$.  The quadratic form that is placed in the $k$-th position among universal quaternary quadratic  forms with discriminant $d$ in Table 3 of \cite{b} is denoted by $Q_d^k$. 
 
Any unexplained notations and terminologies can be found in \cite{ki} or \cite{om}.

\section{A transformation preserving primitively almost universalities}

For a $\z$-lattice $L$, we define
$$
\Lambda_2(L)=\{ \bx \in L : Q(\bx) \equiv 0 \Mod 2\}.
$$
Note that $\Lambda_2(L)$ is, in fact, a $\z$-lattice. We denote by $\lambda_2(L)$ the primitive $\z$-lattice obtained by scaling the quadratic map on $\Lambda_2(L)$ suitably.  

\begin{lem}\label{Watson}  Let $L$ be a primitively almost universal $\z$-lattice such that the rank of the unimodular component of a Jordan decomposition of $L_2$ is less than or equal to $2$. Then  $\lambda_2(L)$ is also primitively almost universal. In particular, if $L$ is primitively universal, then so is $\lambda_2(L)$.  
\end{lem}

\begin{proof}   Since the proof is quite similar to each other, we only deal with the case when  the rank of the unimodular component of a Jordan decomposition of $L_2$ is $2$. In this case,  one may easily show that there is a basis $\{\bx_i\}_{i=1}^k$ of $L$ such that 
$$
B(\bx_i,\bx_j) \equiv 0 \pmod 2 \ \ \text{for any $i,j$ such that $(i,j) \ne (1,1), (2,2)$},
$$
and $B(\bx_1,\bx_1) \equiv B(\bx_2,\bx_2) \equiv 1 \Mod 2$.  
Then we have 
$$
\Lambda_2(L)=\z(2\bx_1)+\z(\bx_1+\bx_2)+\z \bx_3+\dots+\z \bx_k \  \  \text{and} \ \ \mathfrak s(\Lambda_2(L))=2\z.
$$   
For a positive integer $m$, assume that $Q_L(a_1\bx_1+a_2\bx_2+\dots+a_k\bx_k)=2m$ has an integer solution $(a_1,\dots,a_k) \in \z^k$ such that  $(a_1,a_2,\dots,a_k)=1$.  Then from the assumption, we have $a_1\equiv a_2 \Mod 2$. Hence there is an integer $a$ such that $a_1=a_2+2a$.  Therefore we have
$$
Q_{\lambda_2(L)}(a(2\bx_1)+a_2(\bx_1+\bx_2)+a_3\bx_3+\dots+a_k\bx_k)=m, \  \  \text{where $(a,a_2,a_3,\dots,a_k)=1$}.
$$
This implies that $\lambda_2(L)$ represents $m$ primitively. Consequently, we have
$$
E(\lambda_2(L)^*) \subset \{ m : 2m \in E(L^*)\}.
$$
This completes the proof.
\end{proof}

As mentioned in the introduction,  Earnest and Gunawardana proved in \cite{eg} that there are exactly $152$ equivalence classes of primitively almost universal quaternary quadratic forms that are universal. The primitively universalities of $37$  quaternary quadratic forms  listed in Table \ref{table-PU-known} were previously proved in \cite{bu}, \cite{eg}, and \cite{ekm}.   Among $37$ quaternary quadratic forms, there are exactly $10$ non-diagonal quadratic forms, all of whom have class number one.

\begin{table}[h]
	\caption{ Known primitively universal quaternary  forms} 
	\label{table-PU-known}
	\begin{center}
		\small
		\begin{tabular}{llll}
			\hline
			\multicolumn{1}{|l}{$Q_{2}^{1}=(1\,1\,2\,0\,0\,0)$} & $Q_{3}^{1}=(1\,1\,3\,0\,0\,0)$ & $Q_{3}^{2}=(1\,2\,2\,2\,0\,0)$ & \multicolumn{1}{l|}{$Q_{5}^{2}=(1\,2\,3\,2\,0\,0)$} \\
			\multicolumn{1}{|l}{$Q_{6}^{2}=(1\,2\,3\,0\,0\,0)$} & $Q_{7}^{3}=(2\,2\,3\,2\,0\,2)$ & $Q_{8}^{1}=(1\,2\,4\,0\,0\,0)$ & \multicolumn{1}{l|}{$Q_{8}^{2}=(1\,3\,3\,2\,0\,0)$} \\
			\multicolumn{1}{|l}{$Q_{8}^{3}=(2\,2\,2\,0\,0\,0)$} & $Q_{8}^{4}=(2\,2\,3\,2\,2\,0)$ & $Q_{10}^{1}=(1\,2\,5\,0\,0\,0)$ & \multicolumn{1}{l|}{$Q_{12}^{1}=(1\,2\,6\,0\,0\,0)$} \\
			\multicolumn{1}{|l}{$Q_{12}^{2}=(1\,3\,4\,0\,0\,0)$} & $Q_{12}^{3}=(2\,2\,3\,0\,0\,0)$ & $Q_{12}^{5}=(2\,3\,3\,0\,2\,2)$ & \multicolumn{1}{l|}{$Q_{14}^{1}=(1\,2\,7\,0\,0\,0)$} \\
			\multicolumn{1}{|l}{$Q_{15}^{2}=(1\,3\,5\,0\,0\,0)$} & $Q_{15}^{4}=(2\,3\,3\,0\,2\,0)$ & $Q_{18}^{2}=(1\,3\,6\,0\,0\,0)$ & \multicolumn{1}{l|}{$Q_{18}^{4}=(2\,3\,3\,0\,0\,0)$} \\
			\multicolumn{1}{|l}{$Q_{24}^{2}=(2\,2\,6\,0\,0\,0)$} & $Q_{24}^{4}=(2\,3\,4\,0\,0\,0)$ & $Q_{27}^{2}=(2\,3\,5\,0\,2\,0)$ & \multicolumn{1}{l|}{$Q_{28}^{2}=(2\,2\,7\,0\,0\,0)$} \\
			\multicolumn{1}{|l}{$Q_{28}^{3}=(2\,3\,5\,2\,0\,0)$} & $Q_{30}^{1}=(2\,3\,5\,0\,0\,0)$ & $Q_{32}^{1}=(2\,4\,4\,0\,0\,0)$ & \multicolumn{1}{l|}{$Q_{32}^{2}=(2\,4\,5\,4\,0\,0)$} \\
			\multicolumn{1}{|l}{$Q_{40}^{2}=(2\,4\,5\,0\,0\,0)$} & $Q_{42}^{1}=(2\,3\,7\,0\,0\,0)$ & $Q_{48}^{2}=(2\,4\,6\,0\,0\,0)$ & \multicolumn{1}{l|}{$Q_{56}^{1}=(2\,4\,7\,0\,0\,0)$} \\
			\multicolumn{1}{|l}{$Q_{72}^{1}=(2\,4\,9\,0\,0\,0)$} & $Q_{88}^{1}=(2\,4\,11\,0\,0\,0)$ & $Q_{96}^{1}=(2\,4\,12\,0\,0\,0)$ & \multicolumn{1}{l|}{$Q_{104}^{1}=(2\,4\,13\,0\,0\,0)$} \\
			\multicolumn{1}{|l}{$Q_{112}^{1}=(2\,4\,14\,0\,0\,0)$} &  &  & \multicolumn{1}{l|}{} \\
			\hline
		\end{tabular}
	\end{center}
\end{table}

Now, by using Lemma \ref{Watson}, one may easily determine the set of all integers that are primitively represented by each of 18 quaternary forms listed in Table \ref{WatsonT}. Among those forms in Table \ref{WatsonT}, $14$ are primitively universal and the other $4$ have only one exception which is listed in Table \ref{table-APU-Eset}.  These  $18$  primitively almost universal quaternary forms are called quaternary forms of type $0$ in Tables \ref{table-PU} and \ref{table-APU}. 
\begin{table}[h]
\caption{$\lambda_2$-transformations between primitively almost universal quaternary  forms} 
	\label{WatsonT}
	\begin{center}
		\small
	\begin{tabular}{cccc}
		\hline
		\multicolumn{4}{|c|}{Primitively universal quaternary forms in Table \ref{table-PU}}                                                         \\ \hline\hline
		
		\multicolumn{1}{|c}{$\lambda_2(Q_{24}^{6})\simeq Q_{6}^{3}$} & $\lambda_2(Q_{40}^{2})\simeq Q_{10}^{2}$ & $\lambda_2(Q_{40}^{1})\simeq Q_{10}^{3}$ & \multicolumn{1}{c|}{$\lambda_2(Q_{52}^3)\simeq Q_{13}^{2}$}\\
		\multicolumn{1}{|c}{$\lambda_2(Q_{56}^{1})\simeq Q_{14}^{3}$} & $\lambda_2(Q_{68}^3)\simeq Q_{17}^{3}$ & $\lambda_2(Q_{72}^{1})\simeq Q_{18}^{3}$ & \multicolumn{1}{c|}{$\lambda_2(Q_{72}^3)\simeq Q_{18}^{5}$}\\ 	 
		\multicolumn{1}{|c}{$\lambda_2(Q_{80}^3)\simeq Q_{20}^{4}$} & $\lambda_2(Q_{88}^{1})\simeq Q_{22}^{2}$ & $\lambda_2(Q_{88}^{3})\simeq Q_{22}^{4}$ & \multicolumn{1}{l|}{$\lambda_2(Q_{92}^2)\simeq Q_{23}^{2}$}\\ 
		\multicolumn{1}{|c}{ $\lambda_2(Q_{96}^{2})\simeq Q_{24}^{3}$} & $\lambda_2(Q_{104}^{1})\simeq Q_{26}^{2}$  & & \multicolumn{1}{c|}{}\\ 	\hline\hline
		\multicolumn{4}{|c|}{Primitively almost universal quaternary forms in Table \ref{table-APU}}                                                         \\ \hline\hline
		\multicolumn{1}{|c}{$\lambda_2(Q_{24}^1)\simeq Q_{6}^1$}   & $\lambda_2(Q_{28}^1)\simeq Q_{7}^1$ & $\lambda_2(Q_{60}^1)\simeq Q_{15}^3$ & \multicolumn{1}{c|}{$\lambda_2(Q_{80}^1)\simeq Q_{20}^2$} \\ \hline
	\end{tabular}
	\end{center}
\end{table}

\section{Some properties of ternary core $\z$-lattices}

In this and the next sections, we determine the set of all integers that are primitively represented by the remaining $97=152-(37+18)$ primitively almost universal quaternary quadratic forms that are universal. 

First, we introduce the following well known result, which will frequently be used throughout this article. 

\begin{lem} \label{fund} Let $L$ be a $\z$-lattice. If  an integer $m$ is primitively represented by $L$ over $\z_p$ for any prime $p$, then there is a $\z$-lattice $L' \in \gen(L)$  that represents $m$ primitively. In particular, if the class number of $L$ is one, then $L$ itself represents $m$ primitively.     
\end{lem}  

\begin{proof} For the proof, see 102.5 of \cite{om} .
\end{proof}


Let $L$ be a quaternary $\z$-lattice which is isometric to one of $\z$-lattices corresponding to the remaining  $97$ primitively almost universal quaternary quadratic forms that are universal.  Note that neither $L$ is diagonal nor the class number of $L$ is one. 
One may check by direct computations that 
\begin{itemize}
\item [(i)]  if $L$ is isometric to one of $\z$-lattices corresponding to the quaternary quadratic forms given in Table \ref{table-PU}, then it primitively represents all integers less than or equal to $10^5$,  
\item [(ii)]  if $L$ is isometric to one of $\z$-lattices corresponding to the quaternary quadratic forms given in Table \ref{table-APU}, then it primitively represents all integers less than $10^5$, except for those integers given in the same line as $L$ in Table \ref{table-APU-Eset}. 
\end{itemize}
Now, we show that $L$ primitively represents all integers greater than or equal to $10^5$.  Our basic strategy is the following:  Recall that $L \simeq \langle 1\rangle \perp \ell$ for some ternary $\z$-sublattice $\ell$. Let $n$ be a positive integer greater than or equal to $10^5$. To show that $n$ is primitively represented by $L$,  we take a primitive ternary sublattice which is called a core sublattice of $L$, denoted by  $\text{core}(L)$, such that the set $Q(\text{core}(L))$ is completely known in most cases, or at least  the set $Q(\text{core}(L)) \cap A_{u,r}$ is completely known for various integers $u$ and $r$.   
Assume that 
$$
 L=\z \bx_1+\z \bx_2+\z \bx_3+\z \bx_4   \quad \text{and} \quad   \text{core}(L)=\z \bx_1+\z \bx_2+\z \bx_3.
 $$
Let 
$$
\text{core}(L)^{\perp}=\z\left(\sum_{i=1}^4a_i\bx_i\right), \  \ \text{where} \ \  Q\left(\sum_{i=1}^4a_i\bx_i\right)=k,
$$ 
for some positive integers $a_1,a_2,a_3$, and $a_4$. If there are integers $u$ and $b_1,b_2,b_3$ such that 
$$
Q(b_1\bx_1+b_2\bx_2+b_3\bx_3)=n-ku^2 \quad\text{ and}\quad (b_1+ua_1,b_2+ua_2,b_3+ua_3,ua_4)=1,
$$
then clearly, $n$ is primitively represented by $L$. So, it suffices to show that, for any integer $n \ge 10^5$, such integers $b_1,b_2,b_3$, and $u$ satisfying the above condition  always exist. Note that the assumption that $n \ge 10^5$ guarantees that $n-ku^2$ is a positive integer in all cases.  
In most cases, the integer $u$ depends only on the anisotropic primes of $\text{core}(L)$ and the integer $k$. 

We categorize the $97$ quaternary $\z$-lattices into two types as follows:
If $\text{core}(L)$ decomposes $L$ orthogonally,  then  everything is simpler. In this case, we call $L$ a $\z$-lattice of  type $1$ in Tables \ref{table-PU}. There are $9$ $\z$-lattices of type $1$.  In fact, if $L$ is of type $1$ and $\ell$ is the core sublattice of $L$, then we always have $L=\langle 1\rangle \perp \ell$ except for the case when 
$$
L=Q_{96}^2=\langle 1,2\rangle\perp \begin{pmatrix} 4&2\\2&13\end{pmatrix}.
$$  
In this exceptional case, we take $\langle 1\rangle\perp \begin{pmatrix} 4&2\\2&13\end{pmatrix}$ as the core sublattice of $L$.
The remaining $88(=97-9)$ $\z$-lattices are called of type $2$, and they are categorized further according to their core sublattices. We label the core lattices in ascending order of their discriminants by $N_i$ as described in Tables \ref{table-PU} and \ref{table-APU}.

In each case,  the core sublattice $\text{core}(L)$ of $L$ always has class number one except for
$$
\text{core}(L) \simeq  N_9=\begin{pmatrix} 2&1&0\\1&4&2\\0&2&6\end{pmatrix} , \  \  N_4= \begin{pmatrix} 1&0&0\\0&2&1\\0&1&4\end{pmatrix}, \ \ \text{or}  \ \  N_{10}=\begin{pmatrix} 2&1&1\\1&4&0\\1&0&7\end{pmatrix}.
$$
It is well known that $N_9$ is regular, that is, it represents all integers that are locally represented (for this, see \cite{jks}), though it has class number two. Note that
$$
Q(N_9)=\left\lbrace 2m : m\ne 17^{2s+1}m', \  \ \text{for some $s \in \mathbb N_0$ and $m'$ such that $\left (\frac {m'}{17}\right)=-1$} \right\rbrace.
$$

Now, we describe several lemmas on representation properties of core lattices $N_i$, which play crucial roles in Section \ref{section-mainproof} in proving primitively (almost) universalities.

\begin{lem} \label{core1}  For the ternary $\z$-lattice $N_4$ given above, we have
$$
(A_{3,0} \cup A_{3,2})  \cap Q(\gen(N_4)) \subset Q(N_4) \ \ \text{and} \ \ (A_{4,0} \cup A_{4,1}) \cap Q(\gen(N_4)) \subset Q(N_4).
$$
\end{lem} 
\begin{proof}
Since the proofs are quite similar to each other, we only provide the proof of $A_{3,0} \cap Q(\gen(N_4)) \subset Q(N_4)$. 
Note that the class number of $N_4$ is two and  the other class in $\gen(N_4)$ is $N_4'=\mathbb{Z}\bx_1+\mathbb{Z}\bx_2+\mathbb{Z}\bx_3$ whose Gram matrix is $\langle 1,1,7 \rangle$.
 Let $n$ be an integer such that $3n$ is locally represented by $N_4$. Then $3n$ is represented by either $N_4$ or $N_4'$. Assume that $Q(a_1\bx_1+a_2\bx_2+a_3\bx_3)=a_1^2+a_2^2+7a_3^2=3n$. 
Since we may assume, by changing signs if necessary, that $a_1\equiv a_2\equiv a_3\Mod3$, there are integers $a_1'$ and $a_2'$ such that  $a_1=a_3+3a_1'$ and $a_2=a_3+3a_2'$. Therefore we have
$$
Q((a_3+3a_1')\bx_1+(a_3+3a_2')\bx_2+a_3\bx_3)=3n,
$$ 
which implies that $3n$ is represented by
$$
N_4''=\z3\bx_1+\z3\bx_2+\z(\bx_1+\bx_2+\bx_3),
$$ 
whose Gram matrix is $\begin{pmatrix}9&0&3\\0&9&3\\3&3&9\end{pmatrix}$.
Now, one may directly show that $N_4''$ is represented by $N_4$. 
Therefore $3n$ is represented by $N_4$. 
This completes the proof.
\end{proof}

\begin{lem} \label{core2}  For the ternary $\z$-lattice $N_{10}$ given above, we have
$$
(A_{3,0} \cup A_{3,2})  \cap Q(\gen(N_{10})) \subset Q(N_{10}) \ \ \text{and} \ \ (A_{4,0} \cup A_{4,3}) \cap Q(\gen(N_{10})) \subset Q(N_{10}).
$$
\end{lem} 

\begin{proof}
Note that the class number of $N_{10}$ is two and the other class in the genus of $N_{10}$ is $\langle1\rangle\perp\begin{pmatrix}7&2\\2&7\end{pmatrix}$.
Since the remaining of the proof is quite similar to that of the above lemma, it is left as an exercise for the readers. 
\end{proof}

\begin{lem}\label{123}
Let $N_3=\z \bx_1+\z \bx_2+\z \bx_3$ be the ternary $\z$-lattice whose Gram matrix is $\langle 1,2,3\rangle$.
If $n\equiv 4, 6 \Mod 8$, then there is a vector $(a_1,a_2,a_3)\in \z^3$ such that 
$$
Q(a_1\bx_1+a_2\bx_2+a_3\bx_3)=n \quad \text{and} \quad (a_1,a_3)\equiv (1,1) \Mod 2.
$$
\end{lem}

\begin{proof}
Note that the class number of $N_3$ is one, and $N_3$ represents all nonnegative integers except for those integers of the form $2^{2s+1}(8t+5)$ for some nonnegative integers $s$ and $t$. 
Hence if $n\equiv 4 ,6 \Mod 8$, then there is a vector $(a_1,a_2,a_3)\in \z^3$ such that
$Q(a_1\bx_1+a_2\bx_2+a_3\bx_3)=n$. Note that for $t=0$ or $2$, we have 
$$
n \equiv 4+t \Mod 8 \iff (a_1^2,2a_2^2,3a_3^2) \equiv (4,t,0), \ (0,t,4), \ (1,t,3) \Mod 8
$$
Hence  $(a_1,a_3)\equiv (1,1) \Mod 2$ if and only if $a_1\equiv a_3 \Mod 4$ or $a_1 \equiv -a_3 \Mod 4$. Therefore it suffices to show that $n$ is represented by  either
 $$
N_3^{+}=\z (\bx_1+\bx_2)+\z \bx_2+\z (4\bx_3) \ \ \text{or} \ \ N_3^{-}=\z (\bx_1-\bx_2)+\z \bx_2+\z (4\bx_3),
$$
which are sublattices of $N_3$. Since $N_3^{+}\simeq N_3^{-}\simeq \langle 2,4,12\rangle$ and $h(N_3^{\pm})=1$, one may directly check that $n$ is represented by both of them. This completes the proof.
\end{proof}

\begin{lem}\label{lem124first}  Let
 $N_5=\z \bx_1+\z \bx_2+\z \bx_3$ be the ternary $\z$-lattice whose Gram matrix is $\langle1,2,4\rangle$. Assume that an integer $n$ is represented by $N_5$. Then there exist two vectors $(x_1,y_1,z_1),(x_2,y_2,z_2)\in \z^3$ with $z_1$ odd and $z_2$ even such that $n=x_1^2+2y_1^2+4z_1^2=x_2^2+2y_2^2+4z_2^2$ if and only if $n\equiv 4\Mod 8$, $n\equiv 6\Mod {16}$, or $n\equiv 8,16\Mod{32}$.

\end{lem}

\begin{proof}
Note that the class number of $N_5$ is one, and $N_5$ represents all nonnegative integers except for those integers of the form $2^{2s+1}(8t+7)$ for some nonnegative integers $s$ and $t$. 
Assume that $n\equiv 4\Mod 8$, $n\equiv 6 \Mod {16}$, or $n\equiv 8, 16\Mod{32}$. Then, clearly, $n$ is represented by $N_5$.  Hence there is a vector $(a_1,a_2,a_3) \in \z^3$ such that $Q(a_1\bx_1+a_2\bx_2+a_3\bx_3)=n$. 

First, assume that $n\equiv 6 \Mod {16}$. Let $n=2\cdot \tilde{n}$ with $\tilde{n}\equiv 3 \Mod 8$. Then 
we have $a_1\equiv 0 \Mod 2$. Hence there is an integer $a'_1$ such that
$$
a_1=2a'_1 \quad \text{and} \quad 4(a'_1)^2+2a_2^2+4a_3^2=n.
$$
Then we have $2(a'_1)^2+a_2^2+2a_3^2=\tilde{n}$. Since $\tilde{n}\equiv 3 \Mod 8$, we have $(a'_1,a_3)\equiv (1,0)$ or $(0,1) \Mod 2$. Therefore, by substituting $a_3$ by $a'_1$, if necessary, we get the desired result.

Assume that $n\equiv 4\Mod 8$ or $n\equiv 8, 16\Mod{32}$. Then we have $(a_1,a_2)\equiv (0,0)\Mod 2$. Hence there are integers $a'_1$ and  $a'_2$ such that
$$
a_1=2a'_1,\ a_2=2a'_2 \quad \text{and} \quad 4(a'_1)^2+8(a'_2)^2+4a_3^2=n.
$$
If $n\equiv 4 \Mod 8$, then we have $(a'_1,a_3)\equiv (1,0)$ or $(0,1) \Mod 2$. Hence
by substituting $a_3$ by $a'_1$, if necessary, we get the desired result.  Assume that $n\equiv 8 \Mod {32}$. Let $n=4\cdot \tilde{n}$ for some integer $\tilde n$ such that  $\tilde{n}\equiv 2 \Mod 8$. Then we have $(a'_1)^2+2(a'_2)^2+a_3^2=\tilde{n}$.
Note that $a_3\equiv 0 \Mod 2$ if and only if $\tilde{n}$ is represented by $\langle 1,2,4\rangle$, and $a_3\equiv 1 \Mod 2$ if and only if $\tilde{n}$ is represented by $\langle 1,1,8\rangle$. 
Therefore the lemma follows from the fact that $\tilde{n}$ is represented by both   $\langle 1,2,4\rangle$ and $\langle 1,1,8\rangle$.

Assume that $n\equiv 16 \Mod {32}$. Let $n=4\cdot \tilde{n}$ for some integer $\tilde n$ such that $\tilde{n}\equiv 4 \Mod 8$. Then we have $(a'_1)^2+2(a'_2)^2+a_3^2=\tilde{n}$. Now, by Lemma \ref{fund}, we may find an integer solution such that $a_3\equiv 1\Mod 2$.
On the other hand,  there is an integer solution such that $a_3\equiv 0 \Mod 2$ if and only if $\tilde{n}$ is represented by $\langle 1,2,4\rangle$. Therefore the lemma follows directly  from the fact that $\tilde{n}$ is represented by  $\langle 1,2,4\rangle$.

 Conversely, assume that $n\nequiv 4\Mod 8$, $n\nequiv 6 \Mod {16}$, and $n\nequiv 8, 16\Mod{32}$. 
Then we have
$$
n\equiv 1\Mod{2}, \quad n\equiv 2\Mod {8}, \quad n\equiv 0\Mod {32}, \quad \text{or} \quad  n\equiv 24\Mod {64}
$$
from the assumption that $n$ is represented by $N_5$.
If $n\equiv 1,3 \Mod{8}$, $n\equiv 2\Mod {16}$, $n\equiv 0\Mod {32}$, or $n\equiv 24\Mod {64}$, then one may easily check that $a_3$ is always even.
On the other hand, if $n\equiv 5,7\Mod{8}$ or $n\equiv 10\Mod {16}$, then one may easily check that $a_3$ is always odd. 
This completes the proof.
\end{proof}

\begin{lem}\label{125}
Let $N_7=\z \bx_1+\z \bx_2+\z \bx_3$ be the ternary $\z$-lattice whose Gram matrix is $\langle1,2,5\rangle$.
Assume that an integer $n$ is not of the form $5^{2s+1}(5t\pm 2)$ for some nonnegative integers $s$ and $t$.
\begin{itemize}
\item [(i)] If $n\equiv 0 \Mod 8$, then there is a vector $a_1\bx_1+a_2\bx_2+a_3\bx_3 \in N_7$ such that
$$
Q(a_1\bx_1+a_2\bx_2+a_3\bx_3)=n \quad \text{and} \quad (a_1,a_2,a_3) \equiv (1,1,1)  \Mod 2.
$$

\item [(ii)] If $n\equiv 6 \Mod 8$, then there is a vector $b_1\bx_1+b_2\bx_2+b_3\bx_3 \in N_7$ such that
$$
Q(b_1\bx_1+b_2\bx_2+b_3\bx_3)=n \quad \text{and} \quad   (b_1,b_2,b_3) \equiv (1,0,1)  \Mod 2.
$$
\end{itemize}
\end{lem}
\begin{proof}
Since the class number of $N_7$ is one, one may easily check that $n$ is represented by $N_7$ by the assumption. 
If $n\equiv0\Mod8$, then $n$ is primitively represented by $N_7$ over $\z_2$. Hence the proof follows directly from Lemma \ref{fund}.
If $n\equiv 6 \Mod 8$, then the proof follows directly from the fact that $\langle 1,5,8\rangle$ has class number one. This completes the proof.
\end{proof}

\section{Primitively universal quaternary quadratic forms}\label{section-mainproof}

In this section, we prove that each of $97$ quaternary $\z$-lattices of types $1$ and $2$  primitively represents all integers greater than or equal to $10^5$.
In fact, the proofs strongly depend on the ternary core sublattices. If two lattices have the same core sublattice, then the proofs are quite similar to each other. 
So, we only provide the proofs of some representative cases among $\z$-lattices having the same core sublattices. 
For each $\z$-lattice $L$ under consideration, the core sublattice $\text{core}(L)$ is listed in Tables \ref{table-PU} and \ref{table-APU}.

Let $L$ be a $\z$-lattice isometric to one of $97$ $\z$-lattices of types $1$ and $2$ in Tables \ref{table-PU} and \ref{table-APU}, and let $n$ be a positive integer.  Recall that we always assume that 
$$
n \text{\it ~ is greater than or equal to~} 10^5.
$$
 
 First, we consider the case when $L$ is a  $\z$-lattice of type $1$.
 
 \begin{thm} \label{Q_{34}^{3}}
 The quaternary quadratic form $Q_{34}^{3}$ is primitively universal.
 \end{thm} 
 
 \begin{proof}  Let $L=\z \bx_1+\z \bx_2+\z \bx_3+\z \bx_4$ be the quaternary $\z$-lattice corresponding to the quadratic form  $Q_{34}^3$, and let $N_9=\z \bx_2+\z \bx_3+\z \bx_4$ be the ternary sublattice of $L$, which we consider as the core sublattice of $L$,  whose Gram matrices are 
$$
\mathcal M_L=Q_{34}^3=\begin{pmatrix} 1&0&0&0\\0&2&1&0\\0&1&4&2\\0&0&2&6\end{pmatrix}, \qquad \mathcal M_{N_9}=\begin{pmatrix} 2&1&0\\1&4&2\\0&2&6\end{pmatrix},
$$
respectively.
Note that $N_9$ is a regular ternary $\z$-lattice having class number two.
 Hence one may easily check that
$$
A_{2,0}-A_{17,0}\subset Q(\gen(N_9))=Q(N_9).
$$

Let $a_1$ be the positive integer such that $\tilde{n}=n-a_1^2$, where $\tilde{n}$ is defined as
$$
\tilde{n}=\begin{cases}n-1^2\qquad&\text{if $n \equiv 1 \Mod 2,\ n\equiv 0\Mod{17}$},\\
n-17^2\qquad&\text{if $n \equiv 1 \Mod 2,\ n\not\equiv 0\Mod{17}$},\\
n-2^2\qquad&\text{if $n \equiv 0 \Mod 8,\ n\not\equiv 4\Mod{17}$},\\
n-34^2\qquad&\text{if $n \equiv 0 \Mod 8,\ n\equiv 4\Mod{17}$},\\
n-4^2\qquad&\text{if $n \equiv 2, 4, 6 \Mod 8,\ n\not\equiv 16\Mod{17}$},\\
n-8^2\qquad&\text{if $n \equiv 2, 4, 6 \Mod 8,\ n\equiv 16\Mod{17}$}.\end{cases}
$$
Then one may easily check that $\tilde{n}\in A_{2,0}-A_{17,0} \subset Q(N_9)$.
Hence  there are integers $a_2,a_3,a_4$ such that $Q(a_2\bx_2+a_3\bx_3+a_4\bx_4)=\tilde{n}$, and
thus we have $Q(a_1\bx_1+a_2\bx_2+a_3\bx_3+a_4\bx_4)=n$.
Since $a_1 \in \{1,2,4,7,17,34\}$, and 
 $\tilde{n}\not\equiv 0\Mod 8$ if $n\equiv 0\Mod 2$ and  $\tilde{n}\not\equiv 0\Mod{17}$ in all cases,
 one may easily check that $(a_1,a_2,a_3,a_4)=1$. Therefore $n$ is primitively represented by $L$.
 \end{proof}

Now, we consider the quaternary form $Q_{45}^{1}$ which is  of type $1$.  Since $\text{core}(Q_{45}^{1})=N_{10}$ is not regular, the proof is slightly different from that of the above case.     

\begin{thm} \label{Q_{45}^{1}}
The quaternary quadratic form $Q_{45}^1$ is primitively universal. 
\end{thm}
\begin{proof}
Let $L=\z \bx_1+\z \bx_2+\z \bx_3+\z \bx_4$ be the quaternary $\z$-lattice corresponding to  the quadratic form $Q_{45}^1$ and let $N_{10}=\z \bx_2+\z \bx_3+\z \bx_4$ be the ternary sublattice of $L$, which we consider as the core sublattice of $L$,  whose Gram matrices are 
$$
\mathcal M_L=Q_{45}^1=\begin{pmatrix} 1&0&0&0\\0&2&1&1\\0&1&4&0\\0&1&0&7\end{pmatrix}, \qquad \mathcal M_{N_{10}}=\begin{pmatrix} 2&1&1\\1&4&0\\1&0&7\end{pmatrix},
$$
respectively.    Note that $N_{10}$ is universal over $\z_p$ for any prime $p\ne 3,5$. In fact, any integer which is not of the form $3(3t\pm1)$ and  $5^{2s+1}(5t\pm1)$ for some nonnegative integers $s,t$ is represented by the genus of $N_{10}$. 
Furthermore, by Lemma \ref{core2}, we know that
$$
(A_{3,0} \cup A_{3,2})  \cap Q(\gen(N_{10})) \subset Q(N_{10}) \ \ \text{and} \ \ (A_{4,0} \cup A_{4,3}) \cap Q(\gen(N_{10})) \subset Q(N_{10}).
$$

For each $i\in\{1,2\}$, let $a^{(i)}_1$ be the positive integer such that $\widetilde{n_i}=n-\left(a^{(i)}_1\right)^2$, where $\widetilde{n_1}$ and $\widetilde{n_2}$ are defined as
$$
(\widetilde{n_1},\widetilde{n_2})=
\begin{cases} 
(n-1, n-5^2)\equiv (2,2) \Mod 3 \qquad &\text{if $n\equiv 0 \Mod 3$},\\
(n-3^2,n-9^2)\equiv (2,2) \Mod 3 \qquad &\text{if $n\equiv 2 \Mod 3$},\\
(n-3^2,n-9^2)\equiv(0,0) \Mod4 \qquad &\text{if $n\equiv 1 \Mod {12}$},\\
(n-3^2,n-9^2)\equiv(3,3) \Mod{4} \qquad &\text{if $n\equiv 4 \Mod {12}$},\\
(n-6^2,n-12^2)\equiv(3,3) \Mod{4} \qquad &\text{if $n\equiv 7 \Mod {12}$},\\
(n-8^2,n-64^2)\equiv(0,0) \Mod{9} \qquad &\text{if $n\equiv 10 \Mod {36}$},\\
(n-2^2,n-16^2)\equiv(0,0) \Mod{9} \qquad &\text{if $n\equiv 22 \Mod {36}$},\\
(n-4^2,n-32^2)\equiv(0,0) \Mod{9} \qquad &\text{if $n\equiv 34 \Mod {36}$}.
\end{cases} 
$$
Since $(\widetilde{n_1},\widetilde{n_2})\nequiv (0,0)\Mod 5$, and both of them are represented by $N_{10}$ over $\z_2$ and over $\z_3$, at least one of them is represented by $\gen(N_{10})$.
Furthermore, both of them are contained in 
$$
A_{3,0}\cup A_{3,2} \cup A_{4,0} \cup A_{4,3}.
$$
Therefore, either $\widetilde{n_1}$ or $\widetilde{n_2}$ is represented by $N_{10}$ by Lemma \ref {core2}. Equivalently, there are integers $a_2,a_3,a_4$ such that
$Q(a_2\bx_2+a_3\bx_3+a_4\bx_4)=\widetilde{n_i}$ for some $i\in\{1,2\}$.
Furthermore, one may easily check by direct computations that $(a^{(i)}_1,a_2,a_3,a_4)=1$ for any $i=1,2$. Furthermore, since 
$Q(a^{(i)}_1\bx_1+a_2\bx_2+a_3\bx_3+a_4\bx_4)=n$ for some $i=1,2$, the integer $n$ is primitively represented by $L$. This completes the proof. 
\end{proof}

Now, we consider  $\z$-lattices of type $2$  in Tables \ref{table-PU} and \ref{table-APU}.

\begin{thm}\label{Q_{15}^{1}}
The quaternary quadratic form $Q_{15}^{1}$ is primitively universal.
\end{thm}
\begin{proof}
Let $L=\z \bx_1+\z \bx_2+\z \bx_3+\z \bx_4$ be the quaternary $\z$-lattice corresponding to  the quadratic form  $Q_{15}^1$ and let $N_1=\z \bx_1+\z \bx_2+\z \bx_3$ be the ternary sublattice of $L$, which we consider as the core sublattice of $L$, whose Gram matrices are 
$$
\mathcal M_L=Q_{15}^{1}=\begin{pmatrix} 1&0&0&0\\0&1&0&0\\0&0&2&1\\0&0&1&8\end{pmatrix}, \qquad \mathcal M_{N_1}=\begin{pmatrix} 1&0&0\\0&1&0\\0&0&2\end{pmatrix},
$$
respectively. Note that 
$$
N_1^{\perp}=\z(\bx_3-2\bx_4), \ \ \text{where } Q(\bx_3-2\bx_4)=30.
$$ 
For $\bx=a_1\bx_1+a_2\bx_2+a_3\bx_3+ a_4(\bx_3-2\bx_4) \in N_1\perp N_1^{\perp}$, note that if $Q(\bx)=n$ and
\begin{equation}\label{cond:Q_{15}^{1}}
(a_1,a_2,a_3+a_4,-2a_4)=1 
\end{equation}
then $n$ is primitively represented by $L$.   Note also that the class number of $N_1$ is one, and $N_1$ represents all nonnegative integers except for those integers of the form $2^{2s+1}(8t+7)$ for some nonnegative integers $s$ and $t$. 

Let $a_4$ be the positive integer such that $\tilde{n}=n-30\cdot a_4^2$, where $\tilde{n}$ is defined as
$$
\tilde{n}=
\begin{cases} 
n-30 \in \z_2^{\times} \qquad &\text{if $n\equiv 1 \Mod 2$},\\
n-30 \in 2^2(\z_2^{\times}) \qquad &\text{if $n\equiv 2 \Mod 8$},\\
n-30 \equiv 2 \Mod 8 \qquad &\text{if $n\equiv 0 \Mod 8$},\\
n-30 \cdot 2^2 \equiv 4 \Mod 8  \qquad &\text{if $n\equiv 4 \Mod 8$},\\
n-30 \cdot 2^2 \equiv 6 \Mod {16} \qquad &\text{if $n\equiv 14 \Mod {16}$},\\
n-30 \cdot 4^2 \equiv 6 \Mod {16} \qquad &\text{if $n\equiv 6 \Mod {16}$}.
\end{cases} 
$$
Then one may easily show that $\tilde{n}$ is represented by $N_1$.  Hence there is a vector $(a_1,a_2,a_3) \in \z^3$ such that 
$$
Q(a_1\bx_1+a_2\bx_2+a_3\bx_3)=a_1^2+a_2^2+2a_3^2=\tilde n. 
$$
We claim that \eqref{cond:Q_{15}^{1}} holds for all cases. Recall that $a_4 \in \{1,2,4\}$. If $n\equiv 1 \Mod 2$ or $n \equiv 2 \Mod 8$, then we may assume that $a_1 \ \text{or} \ a_2 \equiv 1 \Mod 2$  or  $a_3\equiv 0 \Mod 2$.
If $n \equiv 0 \Mod 8$, then $\tilde n \equiv 2 \Mod 8$. Since $\langle1,1,8\rangle$ represents $\tilde n$,  
 we may assume that $a_3 \equiv 0 \Mod 2$. If $n\equiv 4 \Mod 8$, then by Lemma \ref{fund},  we may assume that  $(a_1,a_2,a_3)\equiv (1,1,1) \Mod 2$.
If $n\equiv 6 \Mod {8}$, then we have $a_3 \equiv 1 \Mod 2$.  From these, one may conclude that \eqref{cond:Q_{15}^{1}} holds for all cases. Therefore the integer $n$ is primitively represented by $L$.  This completes the proof. 
\end{proof}

\begin{thm}\label{Q_{19}^{2}}
The quaternary quadratic form $Q_{19}^{2}$ is primitively universal. 
\end{thm}
\begin{proof}
Let $L=\z \bx_1+\z \bx_2+\z \bx_3+\z \bx_4$ be the quaternary $\z$-lattice corresponding to the quadratic form $Q_{19}^2$ and let $N_3=\z \bx_1+\z \bx_2+\z \bx_3$ be the ternary sublattice of $L$, which we consider as the core sublattice of $L$,  whose Gram matrices are 
$$
\mathcal M_L=Q_{19}^{2}=\begin{pmatrix} 1&0&0&0\\0&2&0&1\\0&0&3&1\\0&1&1&4\end{pmatrix}, \qquad \mathcal M_{N_3}=\begin{pmatrix} 1&0&0\\0&2&0\\0&0&3\end{pmatrix},
$$
respectively. Note that 
$$
N_3^{\perp}=\z(3\bx_2+2\bx_3-6\bx_4), \ \ \text{where } Q(3\bx_2+2\bx_3-6\bx_4)=6\cdot 19=114.
$$ 
For $\bx_\pm=a_1\bx_1+a_2\bx_2+a_3\bx_3\pm a_4(3\bx_2+2\bx_3-6\bx_4) \in N_3\perp N_3^{\perp}$, note that $Q(\bx_+)=Q(\bx_-)$.
Hence if $Q(\bx_\pm)=n$, and either
\begin{equation}\label{cond:Q_{19}^{2}}
(a_1,a_2+3a_4,a_3+2a_4,-6a_4)=1 \quad \text{or}\quad (a_1,a_2-3a_4,a_3-2a_4,+6a_4)=1,
\end{equation}
then $n$ is primitively represented by  $L$.    Note also that the class number of $N_3$ is one, and $N_3$ represents all nonnegative integers except for those integers of the form $2^{2s+1}(8t+5)$ for some nonnegative integers $s$ and $t$. 

Let $a_4$ be the positive integer such that $\tilde{n}=n-114\cdot a_4^2$, where $\tilde{n}$ is defined as
$$
\tilde{n}=
\begin{cases} 
n-114 \in \z_2^{\times} \qquad &\text{if $n\equiv 1 \Mod 2$},\\
n-114 \equiv 6 \Mod 8 \qquad &\text{if $n\equiv 0 \Mod 8$},\\
n-114 \cdot 2^2 \in 2^2(\z_2^{\times}) \qquad &\text{if $n\equiv 4 \Mod 8$},\\
n-114 \in 2^2(\z_2^{\times}) \qquad &\text{if $n\equiv 6 \Mod 8$},\\
n-114 \cdot 2^2 \equiv 2 \Mod {16} \qquad &\text{if $n\equiv 10 \Mod {16}$},\\
n-114 \cdot 4^2 \equiv 2 \Mod {16} \qquad &\text{if $n\equiv 2 \Mod {16}$}.
\end{cases} 
$$
Then one may easily show that $\tilde{n}$ is represented by $N_3$.  Hence there is a vector $(a_1,a_2,a_3) \in \z^3$ such that 
$$
Q(a_1\bx_1+a_2\bx_2+a_3\bx_3)=a_1^2+2a_2^2+3a_3^2=\tilde n. 
$$
We claim that \eqref{cond:Q_{19}^{2}} holds for each case.
If $n\equiv 1 \Mod 2$, then we have $(a_1,a_3)\nequiv (0,0) \Mod 2$. If $n\equiv 0, 4 \Mod 8$, then by Lemma \ref{123}, we may assume that  $(a_1,a_3)\equiv (1,1) \Mod 2$. 
If $n\equiv 6 \Mod 8$, then we have $a_2\equiv 0 \Mod2$. 
 Finally, if $n\equiv 2 \Mod {8}$, then we have $a_2\equiv 1 \Mod2$. From these, one may conclude that \eqref{cond:Q_{19}^{2}} holds for each case. Therefore the integer $n$ is primitively represented by $L$. This completes the proof.  
\end{proof}

\begin{thm} \label{Q_{47}^{1}}
The quaternary quadratic form $Q_{47}^1$ is primitively universal. 
\end{thm}

\begin{proof}  
Let $L=\z \bx_1+\z \bx_2+\z \bx_3+\z \bx_4$ be the quaternary $\z$-lattice corresponding to  the quadratic form $Q_{47}^1$ and let $N_4=\z \bx_1+\z \bx_2+\z \bx_3$ be the ternary sublattice of $L$, which we consider as the core sublattice of $L$,  whose Gram matrices are 
$$
\mathcal M_L=Q_{47}^1=\begin{pmatrix} 1&0&0&0\\0&2&1&0\\0&1&4&1\\0&0&1&7\end{pmatrix}, \qquad \mathcal M_{N_4}=\begin{pmatrix} 1&0&0\\0&2&1\\0&1&4\end{pmatrix},
$$
respectively. Note that 
$$
N_4^{\perp}=\z(\bx_2-2\bx_3+7\bx_4), \ \ \text{where  } Q(\bx_2-2\bx_3+7\bx_4)=7\cdot 47=329.
$$ 
For $\bx_\pm=a_1\bx_1+a_2\bx_2+a_3\bx_3\pm a_4(\bx_2-2\bx_3+7\bx_4) \in N_4\perp N_4^{\perp}$, note that $Q(\bx_+)=Q(\bx_-)$.
Hence if $Q(\bx_\pm)=n$, and either
\begin{equation}\label{cond:Q_{47}^{1}}
	(a_1,a_2+a_4,a_3-2a_4,7a_4)=1 \quad \text{or}\quad (a_1,a_2-a_4,a_3+2a_4,-7a_4)=1,
\end{equation}
then $n$ is primitively represented by  $L$. Note also that  the set $Q(\gen(N_4))$ contains all integers except for those integers of the form $7^{2s+1}t$ for some nonnegative integer $s$ and some positive integer $t$ satisfying $t \equiv 3,5,6 \Mod 7$.  Furthermore,   we have by Lemma \ref{core1},
$$
A_{3,0} \cap Q(\gen(N_4)) \subset Q(N_4) \quad \text{and} \quad  A_{3,2} \cap Q(\gen(N_4)) \subset Q(N_4).
$$

First, assume that $n\equiv 1 \Mod 3$. Let $a_4$ be the positive integer such that $\tilde{n}=n-7\cdot 47 \cdot a_4^2$, where $\tilde{n}$ is defined as
$$\tilde{n}=
\begin{cases} n-7\cdot 47 \not \equiv 0 \Mod 7 \qquad &\text{if $n \not \equiv 0 \Mod 7$,}\\
n-7\cdot47 \in 7(\z_7^{\times})^2 \qquad  &\text{if $n=7k$ and $k\equiv 0,2,6\Mod 7$,}\\
  n-7\cdot47\cdot 2^2 \in 7(\z_7^{\times})^2 \qquad  &\text{if $n=7k$ and $k\equiv 1,3\Mod 7$,}\\
n-7\cdot47\cdot 4^2 \in 7(\z_7^{\times})^2 \qquad  &\text{if $n=7k$ and $k\equiv 4,5\Mod 7$.}\\
 \end{cases}
 $$
Then since $\tilde n \equiv 2 \Mod 3$, $\tilde n$ is represented by $N_4$ by Lemma \ref{core1}.  Hence there is a vector $(a_1,a_2,a_3) \in \z^3$ such that $Q(a_1\bx_1+a_2\bx_2+a_3\bx_3)=\tilde n$. 
Since the class number of $K=\left(\begin{smallmatrix} 2&1\\1&4\end{smallmatrix}\right)$ is one and any even integer is primitively represented by $K$ over $\z_2$, we may assume that $(a_2,a_3) \not \equiv (0,0) \Mod 2$. 
Then one may easily check that \eqref{cond:Q_{47}^{1}} holds for each case and hence $n$ is primitively represented by $L$.

Now, assume that $n \equiv 0 \ \text{or} \ 2 \Mod 3$. In this case, let $a_4$ be the positive integer such that $\tilde{n}=n-7\cdot 47 \cdot a_4^2$, where $\tilde{n}$ is defined as
$$\tilde{n}=
\begin{cases} n-7\cdot 47\cdot 3^2 \not \equiv 0 \Mod 7 \qquad &\text{if $n \not \equiv 0 \Mod 7$,}\\
n-7\cdot47\cdot 3^2 \in 7(\z_7^{\times})^2 \qquad  &\text{if $n=7k$ and $k\equiv 0,4,5\Mod 7$,}\\
  n-7\cdot47\cdot 6^2  \in 7(\z_7^{\times})^2 \qquad  &\text{if $n=7k$ and $k\equiv 2,6\Mod 7$,}\\
n-7\cdot47\cdot 12^2  \in 7(\z_7^{\times})^2 \qquad  &\text{if $n=7k$ and $k\equiv 1,3\Mod 7$.}\\
 \end{cases}
 $$
Then since $\tilde n \equiv 0, 2 \Mod 3$, $\tilde n$ is represented by   $N_4$  by Lemma \ref{core1}. Hence there is a vector  $(a_1,a_2,a_3) \in \z^3$ such that $Q(a_1\bx_1+a_2\bx_2+a_3\bx_3)=\tilde n$. By a similar reasoning to the above, we may assume that $(a_2,a_3) \not \equiv (0,0) \Mod 2$. 
Assume further that  $\tilde n$ is not divisible by $3$, that is, $n$ is not divisible by $3$. Then, clearly, $(a_1,a_2,a_3) \not \equiv (0,0,0) \Mod 3$. 
Thus one may easily check that \eqref{cond:Q_{47}^{1}} holds and hence $n$ is primitively represented by $L$.

Now, assume that $\tilde n$ is divisible by $3$. If $(a_1,a_2,a_3) \not \equiv (0,0,0) \Mod 3$, then the proof is the same to the above.  
Suppose that $(a_1,a_2,a_3) \equiv (0,0,0) \Mod 3$. Since $2a_1^2+(2a_2+a_3)^2+7a_3^2=2\tilde n$ and $9$ is primitively  represented by the binary quadratic form $\langle1,2\rangle$, 
there are integers $b_1,b_2$ such that 
$$
2a_1^2+(2a_2+a_3)^2=2b_1^2+b_3^2 \quad \text{and} \quad b_1b_3 \not \equiv 0 \Mod 3
$$
by Theorem 4.1 of  \cite{oy}.  Furthermore, since $b_3 \equiv a_3 \Mod 2$, there is an integer $b_2$ such that $b_3=2b_2+a_3$. Hence we have
$Q(b_1\bx_1+b_2\bx_2+a_3\bx_3)=\tilde n$.  Now, by using the fact that any even integer is primitively represented by $K$ over $\z_2$, we may further assume that  $(b_2,a_3) \not \equiv (0,0) \Mod 2$. Note that 
$$
Q(b_1\bx_1+(b_2\pm a_4)\bx_2+(a_3\mp2a_4)\bx_3\pm7a_4\bx_4)=n.
$$
Since $a_4 \in \{3,6,12\}$, $b_1 \not\equiv 0 \Mod 3$, $(b_2,a_3)\not\equiv (0,0) \Mod 2$, and $(b_1,b_2,a_3) \not \equiv (0,0,0)\Mod 7$, we have   
$$
(b_1,b_2+a_4,a_3-2a_4,7a_4)=1 \quad \text{or}\quad (b_1,b_2-a_4,a_3+2a_4,-7a_4)=1.
$$
Therefore $n$ is primitively represented by $L$. This completes the proof. 
\end{proof}

Now, we consider the case when the core sublattice is $N_5=\langle 1,2,4\rangle$. There are exactly $29$ primitively universal quaternary $\z$-lattices in Table \ref{table-PU} and 1 primitively almost universal quaternary $\z$-lattice corresponding to the quadratic form  $Q_{80}^1$ in Table \ref{table-APU} whose core sublattice is $N_5$.

\begin{thm} \label{thm124}
Let $a,b$ be nonnegative integers and $c$ be a positive integer such that one of the followings holds;
\begin{enumerate}
\item $b\equiv 1\Mod 2$;
\item $a\equiv c\equiv 0\Mod 2$ and $b\equiv 2\Mod 4$.
\end{enumerate}
Then the quaternary quadratic form $f(x,y,z,w)=x^2+2y^2+4z^2+cw^2+2bzw+2ayw$ primitively represents all  integers greater than $4s^2t$, where $t=4c-2a^2-b^2$, and $s=2$ in the former case and $s=1$ in the latter case.
\end{thm}

\begin{proof}
Let $L=\z \bx_1+\z \bx_2+\z \bx_3+\z \bx_4$ be the quaternary $\z$-lattice corresponding to the quaternary quadratic form $f$ and let $N_5=\z \bx_1+\z \bx_2+\z \bx_3$ be the ternary sublattice of $L$, which we consider as the core sublattice of $L$, whose Gram matrices are
	$$
	\mathcal{M}_L=\begin{pmatrix} 1&0&0&0\\0&2&0&a\\0&0&4&b\\0&a&b&c \end{pmatrix}, \qquad 
	\mathcal{M}_{N_5}=\begin{pmatrix} 1&0&0\\0&2&0\\0&0&4 \end{pmatrix},
	$$

\noindent respectively.
Note that
$$
N_5^{\perp}=\z \left(as \bx_2+\frac{bs}{2} \bx_3-2s\bx_4\right)
$$
and one may easily check that
$$
Q\left(as \bx_2+\frac{bs}{2} \bx_3-2s\bx_4\right)=s^2\left(4c-2a^2-b^2\right)=s^2t \equiv 4\Mod 8.
$$
For $\bx=a_1\bx_1+a_2\bx_2+a_3\bx_3+a_4\left( as \bx_2+\frac{bs}{2} \bx_3-2s\bx_4\right) \in N_5 \perp N_5^{\perp}$, assume that $Q(\bx)=n$.
If $\left( a_1,a_2+asa_4,a_3+\frac{bs}{2}a_4,-2sa_4\right) =1$, then $n$ is primitively represented by $L$.
One may easily check that $as$ is even and $\frac{bs}{2}$ is odd.
Note that the class number of $N_5$ is one and
\begin{equation}\label{Q(1,2,4)}
Q(N_5)=\n-\{2^{2s+1}(8t+7) : s,t\in \n \cup \{0\} \}.	
\end{equation}

Let $n$ be any integer greater than $4s^2t$.
We will show that $n$ is primitively represented by $L$.
Define
$$
\tilde{n}=\begin{cases} n-s^2t\cdot 1^2\qquad &\text{if $n\in A_{2,1}\cup A_{8,0}\cup A_{16,14}$},\\
n-s^2t\cdot 2^2\qquad &\text{if $n\in A_{8,2}\cup A_{8,4}\cup A_{16,6}$}.\end{cases}
$$
Now, by \eqref{Q(1,2,4)}, $\tilde{n}$ is represented by $N_5$ and thus there is a vector $(a_1,a_2,a_3)\in \z^3$ such that $Q(a_1\bx_1+a_2\bx_2+a_3\bx_3)=\tilde{n}$.
From this, it follows that
$$
n=Q\left(a_1\bx_1+\left(a_2+asa_4\right)\bx_2+\left(a_3+\frac{bs}{2}a_4\right)\bx_3-2s a_4\bx_4\right),
$$
where $a_4=1$ or 2.
Hence, to show the primitivity of the above solution, it suffices to show that
$$
\left( a_1,a_2+asa_4,a_3+\frac{bs}{2}a_4,-2sa_4\right) =1.
$$
If $n\equiv 1\Mod 2$, then $\tilde{n}\equiv 1\Mod 2$, which implies that $a_1\equiv 1\Mod 2$.
If $n\equiv 2\Mod 4$, then $\tilde{n}\equiv 2\Mod 4$ and thus $a_2\equiv 1\Mod 2$.
Since $as$ is even, we have $a_2+asa_4\equiv 1\Mod 2$.
If $n\equiv 0\Mod 4$, then $\tilde{n}\equiv 4\Mod 8$.  By Lemma \ref{lem124first}, we may assume that $a_3 \equiv 0 \Mod2$ if $a_4=1$, and $a_3 \equiv 1 \Mod 2$ if $a_4=2$.  Hence we have $a_3+\frac{bs}{2}a_4$ is odd.
This completes the proof.
\end{proof}

\begin{cor}\label{Q_{18}^{5}}  Any quaternary quadratic form in Table \ref{table-PU}  whose core sublattice is $N_5=\langle 1,2,4\rangle$ is primitively universal.  
\end{cor}

\begin{proof} The corollary follows directly from Theorem \ref{thm124}. 
\end{proof}

Now, we provide the proof of the primitively almost universality of $Q_{80}^1=\langle 1,2,4,10\rangle$.  Though the core sublattice of $Q_{48}^1$ is $N_5$, it does not satisfy any condition in Theorem \ref{thm124}. Hence the proof is slightly different from that of the above. 
 
\begin{thm}\label{Q_{80}^{1}}
The quaternary quadratic form $Q_{80}^{1}=\langle1,2,4,10\rangle$ primitively represents all positive integers except for $24$.
\end{thm}
\begin{proof}
Let $L=\z \bx_1+\z \bx_2+\z \bx_3+\z \bx_4$ be the quaternary $\z$-lattice corresponding to the quadratic form $Q_{80}^1$ and let $N_5=\z \bx_1+\z \bx_2+\z \bx_3$ be the ternary sublattice of $L$, which we consider as the core sublattice of $L$, whose Gram matrices are 
$$
\mathcal M_L=Q_{80}^{1}=\begin{pmatrix} 1&0&0&0\\0&2&0&0\\0&0&4&0\\0&0&0&10\end{pmatrix}, \qquad \mathcal M_{N_5}=\begin{pmatrix} 1&0&0\\0&2&0\\0&0&4\end{pmatrix},
$$
respectively. Note that the class number of $N_5$ is one, and $N_5$ represents all nonnegative integers except for those integers of the form $2^{2s+1}(8t+7)$ for some nonnegative integers $s$ and $t$. 

Let $a_4$ be the nonnegative integer such that $\tilde{n}=n-10\cdot a_4^2$, where $\tilde{n}$ is defined as
$$
\tilde{n}=
\begin{cases} 
	n \in \z_2^{\times} \qquad &\text{if $n\equiv 1 \Mod 2$},\\
	n \equiv 2,6,10 \Mod{16} \qquad &\text{if $n\equiv 2 \Mod 4$ and $n \nequiv 14\Mod{16}$},\\
	n-10\cdot 2^2 \equiv 6 \Mod {16} \qquad &\text{if $n\equiv 14 \Mod {16}$},\\
	n-10 \cdot 2^2 \equiv 4 \Mod 8  \qquad &\text{if $n\equiv 4 \Mod 8$},\\
	n-10 \equiv 6 \Mod {16} \qquad &\text{if $n\equiv 0 \Mod {32}$},\\
	n-10 \cdot 2^2 \equiv 16 \Mod {32} \qquad &\text{if $n\equiv 24 \Mod {32}$},\\
	n-10 \cdot 4^2 \equiv 8,16 \Mod {32} \qquad &\text{if $n\equiv 8,16 \Mod {32}$}.
\end{cases} 
$$
Then one may easily show that $\tilde{n}$ is represented by $N_5$.  Hence there are integers $a_1,a_2,a_3$ such that $Q(a_1\bx_1+a_2\bx_2+a_3\bx_3)=a_1^2+2a_2^2+4a_3^2=\tilde n$.

We claim that $a_1,a_2,a_3$ can be taken to satisfy $(a_1,a_2,a_3,a_4)=1$ for each case. If $n\equiv 1 \Mod 2$, $n \equiv 2 \Mod 4$, or $n \equiv 0\Mod{32}$, then $\tilde{n}$ is primitively represented by $N_5$ by Lemma \ref{fund}. 
If $n\equiv 4\Mod{8}$ or $n\equiv 8,16,24\Mod{32}$, then we may assume that $a_3\equiv 1 \Mod{2}$ by Lemma \ref{lem124first}, which implies the claim.
\end{proof}

\begin{thm}\label{Q_{31}^{2}}
The quaternary quadratic form $Q_{31}^{2}$ is primitively universal. 
\end{thm}
\begin{proof}
Let $L=\z \bx_1+\z \bx_2+\z \bx_3+\z \bx_4$ be the quaternary $\z$-lattice corresponding to the quadratic form $Q_{31}^2$ and let $N_6=\z \bx_1+\z \bx_2+\z \bx_4$ be the ternary sublattice of $L$, which we consider as the core sublattice of $L$,  whose Gram matrices are 
$$
\mathcal M_L=Q_{31}^2=\begin{pmatrix} 1&0&0&0\\0&2&1&1\\0&1&4&0\\0&1&0&5\end{pmatrix}, \qquad \mathcal M_{N_6}=\begin{pmatrix} 1&0&0\\0&2&1\\0&1&5\end{pmatrix},
$$
respectively. Note that 
$$
N_6^{\perp}=\z(5\bx_2-9\bx_3-\bx_4), \ \ \text{where  } Q(5\bx_2-9\bx_3-\bx_4)=31\cdot 3^2=279.
$$ 
For $\bx_\pm=a_1\bx_1+a_2\bx_2+a_4\bx_4\pm a_3(5\bx_2-9\bx_3-\bx_4) \in N_6\perp N_6^{\perp}$, note that $Q(\bx_+)=Q(\bx_-)$.
Hence if $Q(\bx_\pm)=n$, and either
\begin{equation}\label{cond:Q_{31}^{2}}
	(a_1,a_2+5a_3,-9a_3,-a_3+a_4)=1 \quad \text{or}\quad (a_1,a_2-5a_3,9a_3,a_3+a_4)=1,
\end{equation}
then $n$ is primitively represented by  $L$.  Note also that the class number of $N_6$ is one, and it represents all integers which are not of the form $2^{2s}(8t+7)$ for some nonnegative integers $s,t$.

Let $a_3$ be the positive integer such that $\tilde{n} = n-279\cdot a_3^2$, where $\tilde{n}$ is defined as
$$
\tilde{n}=
\begin{cases} 
n-279\equiv1\Mod4 \qquad &\text{if $n\equiv 0 \Mod 4$},\\
n-279\equiv2\Mod4 \qquad &\text{if $n\equiv 1 \Mod 4$},\\
n-279\cdot 2^2\equiv2\Mod4 \qquad &\text{if $n\equiv 2 \Mod 4$},\\
n-279\cdot4^2\equiv3\Mod{8} \qquad &\text{if $n\equiv 3 \Mod {8}$},\\
n-279\cdot2^2\equiv3\Mod{8} \qquad &\text{if $n\equiv 7 \Mod {8}$}.\\
\end{cases} 
$$
Then one may easily check that  $\tilde{n}$ is represented by $N_6$. Hence there are integers $a_1, a_2$, and $a_4$ such that  $Q(a_1\bx_1+a_2\bx_2+a_4\bx_4)=\tilde{n}$.
Noting in each case that at least one of $a_1,a_2$, and $a_4$ is odd, one may easily check that \eqref{cond:Q_{31}^{2}} holds. This completes the proof of the theorem.
\end{proof}

\begin{thm}\label{Q_{27}^{3}}
The quaternary quadratic form $Q_{27}^3$ is primitively universal. 
\end{thm} 
 
\begin{proof} Let $L=\z \bx_1+\z \bx_2+\z \bx_3+\z \bx_4$ be the quaternary $\z$-lattice corresponding to the quadratic form $Q_{27}^3$  and let $N_7=\z \bx_1+\z \bx_2+\z \bx_3$ be the ternary sublattice of $L$, which we consider as the core sublattice of $L$, whose Gram matrices are 
$$
\mathcal M_{L}=Q_{27}^3=\begin{pmatrix} 1&0&0&0\\0&2&0&1\\0&0&5&2\\0&1&2&4\end{pmatrix}, \qquad \mathcal M_{N_7}=\begin{pmatrix} 1&0&0\\0&2&0\\0&0&5\end{pmatrix},
$$
respectively. Note that 
$$
N_7^{\perp}=\z(5\bx_2+4\bx_3-10\bx_4), \ \ \text{where  } Q(5\bx_2+4\bx_3-10\bx_4)=10\cdot 27=270.
$$ 
For $\bx_\pm=a_1\bx_1+a_2\bx_2+a_3\bx_3\pm a_4(5\bx_2+4\bx_3-10\bx_4) \in N_7\perp N_7^{\perp}$,  note that $Q(\bx_+)=Q(\bx_-)$.
Hence if $Q(\bx_\pm)=n$, and either
\begin{equation}\label{cond:Q_{27}^{3}}
	(a_1,a_2+5a_4,a_3+4a_4,-10a_4)=1 \quad \text{or}\quad (a_1,a_2-5a_4,a_3-4a_4,10a_4)=1,
\end{equation}
then $n$ is primitively represented by  $L$.   Note also that $N_7$ has class number one and it represents all integers which are not of the form $5^{2s+1}(5t\pm 2)$ for some nonnegative integers $s$ and $t$. Furthermore, for any integer $a$ such that $a \equiv 0 \Mod 8$ and $a \not \equiv 0 \Mod 5$, one may easily show that $a$ is primitively represented by $N_7$ by Lemma \ref{fund}.  

 First, assume that $n \not \equiv 0 \Mod 5$.  We define $\tilde n=n-270$ if $n \not \equiv 0\Mod 8$, and $\tilde n=n-270\cdot 2^2$ otherwise. Let $a_4$ be the positive integer such that $\tilde{n}=n-270\cdot a_4^2$. 
 Since $\tilde n$ is not divisible by $5$, there are integers $a_1,a_2$, and $a_3$ such that
 $$
 Q(a_1\bx_1+a_2\bx_2+a_3\bx_3)=a_1^2+2a_2^2+5a_3^2=\tilde n.
 $$
 We claim that $a_1,a_2,a_3$ can be taken to satisfy \eqref{cond:Q_{27}^{3}} in each case so that $n$ is primitively represented by $L$. Assume that $n$ is not divisible by $8$.
 If $n$ is not divisible by $4$, then either $(a_1,a_3)\nequiv (0,0)\Mod 2$ or $a_2\equiv 0 \Mod 2$, so the claim follows. 
 If $n \equiv 4 \Mod 8$, then $\tilde n \equiv 6 \Mod 8$ and $\tilde n\not \equiv 0 \Mod 5$. Hence, by Lemma \ref{125}, we may find an integer solution of the above equation such that $a_2$ is even, so the claim follows. 
 Finally, if $n$ is divisible by $8$, so is $\tilde n$. Hence, by Lemma \ref{125}, we may find an integer solution of the above equation such that
 $(a_1,a_2,a_3) \equiv (1,1,1) \Mod 2$. This completes the proof of the claim.

Now, assume that $n=5k$ for some integer $k$. Let $a_4$ be the positive integer such that $\tilde{n}=n-270\cdot a_4^2$, where $\tilde{n}$ is defined as
$$
\tilde n=
\begin{cases} n-270 \qquad &\text{if $k\equiv 0 \Mod 5$ and $n \not \equiv 0 \Mod 8$},\\
n-270\cdot 2^2  \qquad &\text{if $k\equiv 0 \Mod 5$ and $n  \equiv 0 \Mod 8$},\\
n-270\cdot 5^2  \qquad &\text{if $k\equiv \pm1 \Mod 5$ and $n \not \equiv 0 \Mod 8$},\\
n-270\cdot 10^2  \qquad &\text{if $k\equiv \pm1 \Mod 5$ and $n  \equiv 0 \Mod 8$},\\
n-270\cdot 3^2  \qquad &\text{if $k\equiv 2 \Mod 5$ and $n \not \equiv 0 \Mod 8$},\\
n-270\cdot 2^2 \qquad &\text{if $k\equiv 2 \Mod 5$ and $n  \equiv 0 \Mod 8$},\\
n-270 \qquad &\text{if $k\equiv 3 \Mod 5$ and $n \not \equiv 0 \Mod 8$},\\
n-270\cdot 6^2 \qquad &\text{if $k\equiv 3 \Mod 5$ and $n  \equiv 0 \Mod 8$}.\\
\end{cases}
$$
In  all cases, $\tilde n$ is divisible by $5$ and $\frac {\tilde{n}}5 \equiv \pm 1 \Mod 5$. Hence $\tilde n$ is represented by $N_7$. Since the proofs are quite similar to each other, we only provide the proof of the fifth case, that is, the case when $\tilde n=n-270\cdot 3^2$. In this case, by using Lemma  \ref{125} and Theorem 4.1 of \cite{oy}, one may find integers $a_1,a_2$, and $a_3$ such that 
  $$
 Q(a_1\bx_1+a_2\bx_2+a_3\bx_3)=a_1^2+2a_2^2+5a_3^2=\tilde n,
 $$
where
 $$
 (a_1,a_3)\nequiv (0,0)\Mod 2 \ \text{or} \  a_2\equiv 0 \Mod 2,  \quad \text{and} \quad (a_1,a_2) \not \equiv (0,0) \Mod 3.
 $$
Since $\tilde n \not \equiv 0 \Mod{25}$, we have $(a_1,a_2,a_3) \not \equiv (0,0,0) \Mod 5$.
From these, one may easily check that \eqref{cond:Q_{27}^{3}} holds, and hence $n$ is primitively represented by $L$. This completes the proof. 
\end{proof}


\newpage 
\begin{center}
	\begin{longtable}{lllc}
		\caption{Primitively universal quaternary  forms and their types}  \label{table-PU}\small\\
			
			\hline
			\multicolumn{4}{|c|}{ \multirow{2}{*}{ Quaternary forms $Q_d^k$ of type $0$}}\\ 
			\multicolumn{4}{|c|}{}\\ \hline\hline

			\multicolumn{1}{|l}{$Q_{6}^{3}=(2\,2\,2\,2\,0\,0)$} & $Q_{10}^{2}=(2\,2\,3\,2\,0\,0)$ & $Q_{10}^{3}=(2\,2\,4\,2\,0\,2)$ & \multicolumn{1}{l|}{$Q_{13}^{2}=(2\,3\,3\,2\,2\,0)$}\\
			\multicolumn{1}{|l}{$Q_{14}^{3}=(2\,2\,4\,2\,0\,0)$} & $Q_{17}^{3}=(2\,3\,4\,0\,2\,2)$ & $Q_{18}^{3}=(2\,2\,5\,2\,0\,0)$ & \multicolumn{1}{l|}{$Q_{18}^{5}=(2\,3\,4\,2\,0\,2)$}\\ 	
			\multicolumn{1}{|l}{$Q_{20}^{4}=(2\,3\,4\,0\,0\,2)$} & $Q_{22}^{2}=(2\,2\,6\,2\,0\,0)$ & $Q_{22}^{4}=(2\,3\,5\,0\,2\,2)$& \multicolumn{1}{l|}{ $Q_{23}^{2}=(2\,3\,5\,2\,0\,2)$}\\ 
			\multicolumn{1}{|l}{$Q_{24}^{3}=(2\,2\,7\,2\,2\,0)$} & $Q_{26}^{2}=(2\,2\,7\,2\,0\,0)$ &  & \multicolumn{1}{l|}{}\\ 	\hline 
			
			&&&\\ \hline

			\multicolumn{3}{|c}{ \multirow{2}{*}{Quaternary forms $Q_d^k$ of type $1$}}& \multicolumn{1}{|c|}{\multirow{2}{*}{Core}}\\ 
			\multicolumn{1}{|c}{}& & & \multicolumn{1}{|c|}{}\\ \hline\hline
			
			\multicolumn{1}{|l}{$Q_{11}^{1}=(1\,2\,6\,2\,0\,0)$} & $Q_{14}^{2}=(1\,3\,5\,2\,0\,0)$  & $Q_{20}^{5}=(2\,4\,4\,4\,2\,0)$ & \multicolumn{1}{|c|}{\multirow{3}{*}{$\langle 1 \rangle^\perp$ ($\ast$ $\langle 2\rangle ^\perp$)}}\\
			\multicolumn{1}{|l}{$Q_{24}^{5}=(2\,4\,4\,0\,2\,2)$} & $Q_{39}^{1}=(2\,3\,7\,0\,2\,0)$ & $Q_{40}^{3}=(2\,4\,6\,2\,0\,2)$  & \multicolumn{1}{|c|}{} \\
			\multicolumn{1}{|l}{$Q_{34}^{3}=(2\,4\,6\,4\,0\,2)$} & $Q_{45}^{1}=(2\,4\,7\,0\,2\,2)$ & $Q_{96}^{2}=(2\,4\,13\,4\,0\,0)^\ast$ & \multicolumn{1}{|c|}{} \\ \hline 
			
			&&&\\ \hline
			
			\multicolumn{3}{|c}{\multirow{2}{*}{Quaternary forms $Q_d^k$ of type $2$}} & \multicolumn{1}{|c|} {\multirow{2}{*}{Core}}\\ 
			\multicolumn{1}{|c}{}& & &\multicolumn{1}{|c|}{}\\ \hline \hline
			
			\multicolumn{1}{|l}{$Q_{7}^{2}=(1\,2\,4\,2\,0\,0)$} & $Q_{15}^{1}=(1\,2\,8\,2\,0\,0)$& & \multicolumn{1}{|c|}{$N_1=\langle 1,1,2 \rangle$}\\  \hline\hline
			
			\multicolumn{1}{|l}{$Q_{11}^{2}=(1\,3\,4\,2\,0\,0)$}& $Q_{17}^{2}=(1\,3\,6\,2\,0\,0)$& & \multicolumn{1}{|c|}{$N_2=\langle 1,1,3 \rangle$} \\  \hline\hline			           
			
			\multicolumn{1}{|l}{$Q_{19}^{2}=(2\,3\,4\,2\,2\,0)$} & $Q_{31}^{1}=(2\,3\,6\,2\,2\,0)$ & $Q_{34}^{1}=(2\,3\,6\,2\,0\,0)$ & \multicolumn{1}{|c|}{\multirow{2}{*}{$N_3=\langle 1,2,3 \rangle$}}\\
			
			\multicolumn{1}{|l}{$Q_{40}^{1}=(2\,3\,7\,2\,0\,0)$}& $Q_{43}^{1}=(2\,3\,8\,2\,2\,0)$&$Q_{46}^{1}=(2\,3\,8\,2\,0\,0)$ & \multicolumn{1}{|c|}{}\\	\hline\hline
			
			\multicolumn{1}{|l}{$Q_{28}^{4}=(2\,4\,4\,0\,2\,0)$} &$Q_{35}^{1}=(2\,4\,5\,0\,0\,2)$ &$Q_{41}^{1}=(2\,4\,7\,4\,0\,2)$ & \multicolumn{1}{|c|}{\multirow{2}{*}{$N_4=\langle 1\rangle \perp \left(\begin{smallmatrix}2&1\\1&4\end{smallmatrix}\right)$}}\\
			\multicolumn{1}{|l}{$Q_{42}^{2}=(2\,4\,6\,0\,0\,2)$}&  $Q_{47}^{1}=(2\,4\,7\,2\,0\,2)$ & &\multicolumn{1}{|c|}{}\\	\hline\hline
			
			\multicolumn{1}{|l}{$Q_{22}^{3}=(2\,3\,4\,2\,0\,0)$} & $Q_{24}^{6}=(2\,4\,4\,4\,0\,0)$ & $Q_{26}^{3}=(2\,4\,4\,2\,2\,0)$ & \multicolumn{1}{|c|}{\multirow{11}{*}{$N_5=\langle 1,2,4 \rangle$}} \\
			\multicolumn{1}{|l}{$Q_{30}^{2}=(2\,4\,4\,2\,0\,0)$} & $Q_{34}^{2}=(2\,4\,5\,2\,2\,0)$ & $Q_{38}^{1}=(2\,4\,5\,2\,0\,0)$ & \multicolumn{1}{|c|}{}\\
			\multicolumn{1}{|l}{$Q_{40}^{4}=(2\,4\,6\,4\,0\,0)$} & $Q_{42}^{3}=(2\,4\,6\,2\,2\,0)$ & $Q_{46}^{2}=(2\,4\,6\,2\,0\,0)$ & \multicolumn{1}{|c|}{}\\
			\multicolumn{1}{|l}{$Q_{50}^{1}=(2\,4\,7\,2\,2\,0)$} &  $Q_{54}^{2}=(2\,4\,7\,2\,0\,0)$ & $Q_{56}^{2}=(2\,4\,8\,4\,0\,0)$ & \multicolumn{1}{|c|}{}\\
			\multicolumn{1}{|l}{$Q_{58}^{2}=(2\,4\,8\,2\,2\,0)$} & $Q_{62}^{1}=(2\,4\,8\,2\,0\,0)$ & $Q_{66}^{1}=(2\,4\,9\,2\,2\,0)$ & \multicolumn{1}{|c|}{}\\
			\multicolumn{1}{|l}{$Q_{70}^{1}=(2\,4\,9\,2\,0\,0)$} &  $Q_{72}^{2}=(2\,4\,10\,4\,0\,0)$ & $Q_{74}^{1}=(2\,4\,10\,2\,2\,0)$ & \multicolumn{1}{|c|}{}\\
			\multicolumn{1}{|l}{$Q_{78}^{1}=(2\,4\,10\,2\,0\,0)$} & $Q_{82}^{1}=(2\,4\,11\,2\,2\,0)$ & $Q_{86}^{1}=(2\,4\,11\,2\,0\,0)$ &\multicolumn{1}{|c|}{}\\
			\multicolumn{1}{|l}{$Q_{88}^{2}=(2\,4\,12\,4\,0\,0)$} & $Q_{90}^{1}=(2\,4\,12\,2\,2\,0)$ & $Q_{94}^{1}=(2\,4\,12\,2\,0\,0)$ & \multicolumn{1}{|c|}{}\\
			\multicolumn{1}{|l}{$Q_{98}^{1}=(2\,4\,13\,2\,2\,0)$} & $Q_{102}^{1}=(2\,4\,13\,2\,0\,0)$ & $Q_{104}^{2}=(2\,4\,14\,4\,0\,0)$ & \multicolumn{1}{|c|}{}\\
			\multicolumn{1}{|l}{$Q_{106}^{1}=(2\,4\,14\,2\,2\,0)$} & $Q_{110}^{1}=(2\,4\,14\,2\,0\,0)$ & &\multicolumn{1}{|c|}{}\\ \hline\hline

			\multicolumn{1}{|l}{$Q_{31}^{2}=(2\,4\,5\,0\,2\,2)$} & & &\multicolumn{1}{|c|}{$N_6=\langle 1\rangle \perp \left(\begin{smallmatrix}2&1\\1&5\end{smallmatrix}\right)$}\\	\hline\hline

			\multicolumn{1}{|l}{$Q_{27}^{3}=(2\,4\,5\,4\,0\,2)$} & & & \multicolumn{1}{|c|}{$N_7=\langle 1,2,5\rangle$}\\	\hline\hline

			\multicolumn{1}{|l}{$Q_{38}^{2}=(2\,4\,6\,0\,2\,2)$} & & &\multicolumn{1}{|c|}{$N_8=\langle 1\rangle \perp \left(\begin{smallmatrix}2&1\\1&6\end{smallmatrix}\right)$} \\	\hline
			
	\end{longtable}
\end{center}

\begin{table}[h]
	\caption{Primitively almost universal quaternary forms that are universal and their types} 
	\label{table-APU}
	\begin{center}
		\small
		\begin{tabular}{lllc}
			
			\hline
			\multicolumn{4}{|c|}{ \multirow{2}{*}{Quaternary forms $Q_d^k$ of type $0$}}\\ 
			\multicolumn{4}{|c|}{}\\ \hline\hline
			
			\multicolumn{1}{|l}{$Q_{6}^1=(1\,1\,6\,0\,0\,0)$} & $Q_{7}^1=(1\,1\,7\,0\,0\,0)$ & $Q_{15}^3=(2\,2\,5\,0\,0\,2)$ & \multicolumn{1}{l|}{$Q_{20}^2=(2\,2\,5\,0\,0\,0)$} \\ \hline
			&&&\\ 
			
			\hline
			\multicolumn{3}{|c}{\multirow{2}{*}{Quaternary forms $Q_d^k$ of type $2$}}  & \multicolumn{1}{|c|}{\multirow{2}{*}{Core}}\\ 
			\multicolumn{3}{|c}{}& \multicolumn{1}{|c|}{}\\ \hline \hline
			
			\multicolumn{1}{|l}{$Q_{18}^1=(1\,2\,9\,0\,0\,0)$} & $Q_{19}^1=(1\,2\,10\,2\,0\,0)$ & $Q_{22}^1=(1\,2\,11\,0\,0\,0)$ &  \multicolumn{1}{|c|}{\multirow{3}{*}{$N_1=\langle 1,1,2 \rangle$}}\\
			\multicolumn{1}{|l}{$Q_{23}^1=(1\,2\,12\,2\,0\,0)$ } & $Q_{24}^1=(1\,2\,12\,0\,0\,0)$ &$Q_{26}^1=(1\,2\,13\,2\,0\,0)$ &  \multicolumn{1}{|c|}{}\\  
			\multicolumn{1}{|l}{$Q_{27}^1=(1\,2\,14\,2\,0\,0)$} & $Q_{28}^1=(1\,2\,14\,0\,0\,0)$ & & \multicolumn{1}{|c|}{} \\\hline\hline
			
			\multicolumn{1}{|l}{$Q_{51}^1=(2\,3\,9\,0\,2\,0)$} & $Q_{54}^1=(2\,3\,9\,0\,0\,0)$ & $Q_{55}^1=(2\,3\,10\,2\,2\,0)$ & \multicolumn{1}{|c|}{\multirow{2}{*}{$N_3=\langle 1,2,3 \rangle$}}\\
			\multicolumn{1}{|l}{$Q_{58}^1=(2\,3\,10\,2\,0\,0)$} & $Q_{60}^1=(2\,3\,10\,0\,0\,0)$ & & \multicolumn{1}{|c|}{}\\  \hline\hline			           
			
			\multicolumn{1}{|l}{$Q_{80}^1=(2\,4\,10\,0\,0\,0)$} & & & \multicolumn{1}{|c|}{{$N_5=\langle 1,2,4 \rangle$}}\\	\hline\hline
			
			\multicolumn{1}{|l}{$Q_{46}^3=(2\,5\,6\,4\,0\,2)$} & $Q_{54}^3=(2\,5\,6\,0\,0\,2)$ & $Q_{54}^4=(2\,5\,7\,4\,2\,2)$ & \multicolumn{1}{|c|}{\multirow{2}{*}{$N_6=\langle 1\rangle \perp \left(\begin{smallmatrix}2&1\\1&5\end{smallmatrix}\right)$}}\\
			\multicolumn{1}{|l}{$Q_{55}^3=(2\,5\,7\,4\,0\,2)$} & $Q_{58}^4=(2\,5\,7\,0\,2\,2)$ & $Q_{63}^1=(2\,5\,7\,0\,0\,2)$ & \multicolumn{1}{|c|}{}\\ \hline\hline

			\multicolumn{1}{|l}{$Q_{42}^4=(2\,5\,5\,4\,0\,0)$} &   $Q_{47}^2=(2\,5\,6\,4\,2\,0)$ & $Q_{48}^3=(2\,5\,5\,2\,0\,0)$ & \multicolumn{1}{|c|}{\multirow{7}{*}{$N_7=\langle 1,2,5 \rangle$}} \\
			\multicolumn{1}{|l}{$Q_{52}^3=(2\,5\,6\,4\,0\,0)$} & $Q_{55}^2=(2\,5\,6\,0\,2\,0)$& $Q_{58}^3=(2\,5\,6\,2\,0\,0)$ &  \multicolumn{1}{|c|}{}\\
			\multicolumn{1}{|l}{$Q_{60}^3=(2\,5\,6\,0\,0\,0)$} & $Q_{62}^2=(2\,5\,7\,4\,0\,0)$ & $Q_{63}^2=(2\,5\,7\,2\,2\,0)$& \multicolumn{1}{|c|}{}\\
			\multicolumn{1}{|l}{$Q_{68}^3=(2\,5\,7\,2\,0\,0)$} & $Q_{70}^2=(2\,5\,7\,0\,0\,0)$& $Q_{72}^3=(2\,5\,8\,4\,0\,0)$&  \multicolumn{1}{|c|}{}\\
			\multicolumn{1}{|l}{$Q_{78}^2=(2\,5\,8\,2\,0\,0)$} & $Q_{80}^3=(2\,5\,8\,0\,0\,0)$ & $Q_{82}^2=(2\,5\,9\,4\,0\,0)$& \multicolumn{1}{|c|}{}\\
			\multicolumn{1}{|l}{$Q_{87}^1=(2\,5\,10\,4\,2\,0)$} & $Q_{88}^3=(2\,5\,9\,2\,0\,0)$ & $Q_{90}^2=(2\,5\,9\,0\,0\,0)$ & \multicolumn{1}{|c|}{}\\
			\multicolumn{1}{|l}{$Q_{92}^2=(2\,5\,10\,4\,0\,0)$} & $Q_{95}^1=(2\,5\,10\,0\,2\,0)$ & $Q_{98}^2=(2\,5\,10\,2\,0\,0)$ & \multicolumn{1}{|c|}{}\\ 
			\hline
		\end{tabular}
	\end{center}
\end{table}

\begin{table}[h]
	\caption{$E((Q_d^k)^*)$ for primitively almost universal quaternary forms $Q_d^k$ that are universal} 
	\label{table-APU-Eset}
	\begin{center}
	\begin{tabular}{|l|l|c|}
		\hline
		\multicolumn{1}{|c|}{$E((Q_d^k)^*)$}               & \multicolumn{1}{c|}{Primitively almost universal quaternary forms $Q_d^k$}  & $\#$ \\ \hline \hline
		\multicolumn{1}{|c|}{\multirow{2}{*}{\{4\}}}   & $Q_{6}^1\ $ $Q_{7}^1\ $ $Q_{15}^3$ $Q_{42}^4$ $Q_{46}^3$ $Q_{47}^2$ $Q_{48}^3$ $Q_{52}^3$ $Q_{54}^4$ $Q_{54}^3$ $Q_{55}^2$ $Q_{55}^3$ $Q_{58}^3$ & \multirow{2}{*}{26}\\
		& $Q_{58}^4$ $Q_{60}^3$ $Q_{62}^2$ $Q_{63}^1$ $Q_{68}^3$ $Q_{70}^2$ $Q_{72}^3$ $Q_{78}^2$ $Q_{82}^2$ $Q_{88}^3$ $Q_{90}^2$ $Q_{92}^2$ $Q_{98}^2$ & \\ \hline

		\multicolumn{1}{|c|}{\multirow{1}{*}{\{8\}}}   & $Q_{18}^1$ $Q_{19}^1$ $Q_{22}^1$ $Q_{23}^1$ $Q_{24}^1$ $Q_{26}^1$ $Q_{27}^1$ $Q_{28}^1$ $Q_{51}^1$ $Q_{54}^1$ $Q_{55}^1$ $Q_{58}^1$ $Q_{60}^1$    & 13 \\ \hline
		
		\multicolumn{1}{|c|}{\{12\}}                   & $Q_{20}^2$ & 1 \\ \hline
		\multicolumn{1}{|c|}{\{24\}}                   & $Q_{80}^1$ & 1 \\ \hline
		\multicolumn{1}{|c|}{\{4,25\}}                 & $Q_{63}^2$ $Q_{87}^1$ & 2 \\ \hline
		\multicolumn{1}{|c|}{\{4,68\}}                 & $Q_{80}^3$ &1  \\ \hline
		\multicolumn{1}{|c|}{\{4,12,25\}}              & $Q_{95}^1$ &1\\ \hline
	\end{tabular}
\end{center}
\end{table}

\end{document}